\documentclass[11pt,oneside,leqno]{amsart}

\usepackage{amsmath}
\usepackage{amsfonts}
\usepackage{amssymb}
\usepackage{amsthm}
\usepackage{amsxtra}
\usepackage{amsopn}
\usepackage{amscd}
\usepackage{pifont}
\usepackage{latexsym}
\usepackage{verbatim}
\usepackage{pb-diagram}
\usepackage{subfig}

\usepackage{fancyhdr}
\newcommand{\gt}{{g}_\tau}

\newcommand{\htt}{{H}_3(\tau)}
\newcommand{\nt}{{\nabla}^\tau}

\newcommand{\Rt}{{R}^\tau}
\newcommand{\lm}{\e}

\newcommand\R{{\mathbb R}}

\newcommand{\tq}{\tau^2}
\newcommand{\sh}{\sinh \vartheta}
\newcommand{\ch}{\cosh \vartheta}
\newcommand{\cph}{\cosh \varphi}
\newcommand{\sph}{\sinh \varphi}

\theoremstyle{plain}
\newtheorem{theorem}{Theorem}[section]
\newtheorem*{theorem*}{Theorem}
\newtheorem{defn}[theorem]{Definition}
\newtheorem{lemma}[theorem]{Lemma}
\newtheorem{prop}[theorem]{Proposition}
\newtheorem{cor}[theorem]{Corollary}
\newtheorem{rem}[theorem]{Remark}

\newtheorem*{mt*}{Main Theorem}
\newcommand\e{{\varepsilon}}

\newcommand\pa[1]{\partial_{#1}}

\begin{document}

\title[]{Helix surfaces in Lorentzian Heisenberg group}

\author{Lorenzo Pellegrino}
\address{Lorenzo Pellegrino: Dipartimento di Matematica e Fisica \lq\lq E. De Giorgi\rq\rq \\
Universit\`a del Salento\\
Prov. Lecce-Arnesano \\
73100 Lecce\\ Italy.}
\email{lorenzo.pellegrino@unisalento.it}

\subjclass[2020]{53C50, 53C25}
\keywords{Lorentzian Lie groups, Lorentzian Heisenberg group, Helix surfaces, constant angle surfaces, Minimal and CMC surfaces}

\begin{abstract}
In this work we investigate constant angle surfaces in the Lorentzian Heisenberg group $\htt$. After providing a complete description of the geometry of the ambient space, we perform the full classification of minimal and CMC helix surfaces in $\htt$, giving their  explicit parametrizations. In addition, we investigate the constant angle spacelike and timelike surfaces for a Lorentzian metric on the Heisenberg group $H_3$ not treated before in the literature, first showing that such surfaces have constant Gaussian curvature and then obtaining  their complete characterization.
\end{abstract}

\maketitle

\section{Introduction}
\setcounter{equation}{0}

The investigation of surfaces in three-dimensional pseudo-Riemannian manifolds is a highly active research area, as their description is crucial in many fields of science. Moreover, the study of submanifolds generally leads to a deeper understanding of the ambient spaces in which they are embedded.

Special classes of surfaces in pseudo-Riemannian three-manifolds may be defined requiring some conditions on either the second fundamental forms of immersions, the shape operator or  the mean curvature. Examples can be found in \cite{CV1, CV4, IV, TA} and references therein. In a series of recent papers (\cite{CP0}--\cite{CP3}), the present author and G. Calvaruso investigated parallel, totally geodesic and totally umbilical surfaces in three-dimensional Lorentzian manifolds admitting a four-dimensional isometry group. These works also provided some explicit examples of minimal and constant mean curvature (CMC) surfaces.

This paper further develops the line of research in this direction including the study of some minimal, constant mean curvature (CMC), and parallel surfaces in a significant Lorentzian Lie group. Particular attention is given to establishing connections and analogies with another important class of surfaces, namely the \textit{helix surfaces} or \textit{constant angle surfaces}. These surfaces are obtained if we require that a unit normal vector field forms a constant angle with a fixed vector field of the ambient space. This geometric requirement has a large interest in many field of application such as Theoretical Physics, Engineering and Computer Science. However, the constancy of some angle on a surface also models interesting situations from daily life, infact a classical example is given by spiral staircases, which correspond to helix surfaces in Euclidean space.

In recent years, helix surfaces have been studied in many Riemannian ambient spaces (see, for example, \cite{CD}--\cite{DMVV},\cite{FMV},\cite{LM},\cite{LM2},\cite{MO}--\cite{Ni} and references therein). Furthermore, the investigation of such surfaces also extended to other settings,  such as higher codimension (\cite{DR},\cite{DR2},\cite{Ru}) and some Lorentzian ambient space. In fact, when we take into account Lorentzian settings, we are allowed to consider more possibilities, as both spacelike and timelike surfaces can be studied. In particular, in the Lorentzian framework, constant angle surfaces have been described in \cite{COPU},\cite{FN},\cite{LM},\cite{LO},\cite{OPP} and \cite{P}.

In particular, in \cite{OP} the authors characterized constant angle surfaces in the {\em Heisenberg group} $H_3$ equipped with a Lorentzian metric. More in general, $H_3$ is the subgroup of $GL(3,\mathbb{R})$ consisting of matrices of the form
$$
\left(
\begin{array}{ccc}
 1 & x & y \\ 0 & 1 & z \\ 0 & 0 & 1 
\end{array}
\right).
$$

The Lorentzian space considered in \cite{OP} is obtained as the restriction $\mathbb{L}^3(0, \tau)$ for $\kappa=0$ of the family $\mathbb{L}^3(\kappa, \tau)$, that is, the original analogue in Lorentzian framework of the so-called Bianchi Cartan Vranceanu spaces (LBCV spaces). In a recent work \cite{CP1}, the present author and G. Calvaruso generalized the geometric construction of these homogeneous Lorentzian three-manifolds considering a new original case with some interesting properties. Thus, another Lorentzian metric appears for the case of Heisenberg group $H_3$, fulfilling the correspondence to the two possibly non-flat Lorentzian metric associated to $H_3$   described in \cite{RR1} and \cite{RR2}.

This paper is devoted to the study of helix surfaces in the Lorentzian Heisenberg group given by $\mathbb{M}^3_1(0,\tau)$ (with $\tau\neq 0$), denoted by $\htt$. In Section 2 we include some useful preliminary result on the geometry of $\htt$ and some basic properties of surfaces in three-dimensional Lorentzian manifolds. Then, in Section 3, we give the Gauss and Codazzi equations for an oriented surface $M$ immersed in $\htt$ to specialize, in Section 4,  to the case of helix surface. In particular, we give the classification of helix surfaces that are minimal and CMC, showing some link with the parallel ones. Finally, in the last section, we treat both timelike and spacelike helix surfaces, characterizing in both cases the non CMC constant angle surfaces and giving an explicit parametrization and some examples.

\section{Preliminaries}
\subsection{Basic geometry of the Lorentzian Heisenberg group}
As we mentioned in the Introduction, furnishing the Heisenberg group of a Lorentzian metric is a well-known problem and it involves also the description of Lorentzian spaces whose isometry group has dimension four. In particular, the two non-flat Lorentzian metrics on $\htt$ are a special case of the more general Lorentzian spaces that takes their name from Bianchi, Cartan and Vranceanu, who studied their Riemannian analogue.
 
LBCV spaces can be described via fibrations over complete, simply connected surfaces of constant curvature $\kappa \in \R$, either Riemannian or Lorentzian, with geodesic fibers. For each of such a surface, there exists a one-parameter family $\mathbb{M}^3_1(\kappa, \tau)$ of fibrations, parametrized by the bundle curvature $\tau \in \R - {0}$. To obtain a Lorentzian metric, the submersion must be Lorentzian when the base is Riemannian, and \textit{viceversa}.

We use standard coordinates $(x, y, z)$ on $\R^3$ and define $\mathbb{M}^3_1(\kappa, \tau) = (\R^3, g_{\kappa,\tau})$, where the metric $g_{\kappa,\tau}$ is given by
\begin{equation}\label{gkt}
g_{\kappa,\tau}=\frac{dx^2 - \delta  dy^2}{\left(1 + \frac{\kappa}{4}(x^2 - \delta  y^2)\right)^2} + \delta \left(dz + \tau \frac{y  dx - x  dy}{1 + \frac{\kappa}{4}(x^2 - \delta  y^2)} \right)^2, \qquad \delta = \pm 1.
\end{equation}
The case $\delta = -1$ corresponds to Lorentzian spaces studied in \cite{Y}. While, the case $\delta = 1$, representing a Riemannian submersion over a Lorentzian surface $\mathbb{M}^2_1(\kappa)$, has been considered for the first time by the present author and Calvaruso in \cite{CP1}.

Observe that, when $\kappa=0$, the space $\mathbb{M}^3_1(0,\tau)$ is the Lorentzian Heisenberg space furnished with two families of non-flat, left-invariant Lorentzian metrics depending on $\delta$, isometric to the non flat-metrics given in \cite{RR1}.

\begin{rem}{\em
Moreover, we point out that:
\begin{itemize}
\item $\mathbb{M}^3_1(0,0) = \R^3_1$ is Minkowski space;
\item $\mathbb{M}^3_1(\kappa, 0)$ is isometric to a Lorentzian product $\mathcal{M}_\delta$, where $\mathcal{M}_{-1} = \mathbb{M}^2(\kappa) \times \R_1$ and $\mathcal{M}_1 = \mathbb{M}^2_1(\kappa) \times \R$;
\item $\mathbb{M}^3_1(-4\tau^2, \tau)$ has constant negative sectional curvature.
\end{itemize}}
\end{rem}

As in the Riemannian setting, for $\tau \neq 0$, the isometry group has two connected components in both cases $\delta = \pm 1$.

From now on we shall denote by $\htt$ ($\tau \neq 0$) the 3-dimensional Heisenberg
group given by $\mathbb{R}^3$ equipped with one of the two non-isometric 1-parameter family of Lorentzian metrics obtained requiring $\kappa=0$ in \eqref{gkt}:

\begin{equation}\label{gonH}
\gt=dx^2-\delta dy^2+\delta(dz+\tau(ydx-xdy))^2.
\end{equation}

Now, let us consider the vector fields:
\begin{align}\label{E123}
E_1=\frac{\partial}{\partial x}+\tau y\frac{\partial}{\partial z}, \qquad E_2=\frac{\partial}{\partial y}-\tau x\frac{\partial}{\partial z}, \qquad
E_3=\frac{\partial}{\partial z},
\end{align}
which form a Lorentzian orthonormal basis on $\htt$, with  $\gt(E_1, E_1)=1$, $\gt(E_3,E_3)=-\gt(E_2,E_2)=\delta$. 

Moreover, for any $\delta$ we have
\begin{equation}\label{Liealg}
[E_1, E_2]=2\tau E_3, \quad [E_2, E_3]=[E_3, E_1]=0.
\end{equation}

At this point, we are able to calculate the associated Levi-Civita connection $\nt$:
\begin{align}\label{LCivitahtt}
\begin{array}{ll}
        \nt_{E_1}{E_1}=\nt_{E_2}E_2=\nt_{E_3}E_3=0, \quad
        &\nt_{E_2}E_1=-\nt_{E_1}E_2=-\tau E_3, \\[3pt]
        \nt_{E_3} E_1=\nt_{E_1}E_3=\tau E_2, \quad
        &\nt_{E_3} E_2=\nt_{E_2} E_3=\delta\tau E_1.
        \end{array}
\end{align}
We may observe that $E_3$ is a $\delta$-unit non degenerate vector field tangent to the fibers of the submersion $\pi:\htt \rightarrow \R^2$  (respectively, $\pi:\htt \rightarrow \R^2_1$) for $\delta=-1$ (respectively, $\delta=1$). By \eqref{LCivitahtt}, we have
\begin{equation}\label{nxe3htt}
    \nt_X E_3=\delta \tau X \wedge E_3 \quad \forall X\in \mathfrak{X}(\htt),
\end{equation}
where $\wedge$ is the cross product in $\htt$ defined by 
$$
E_1\wedge E_2=\delta E_3, \quad E_2\wedge E_3=E_1,  \quad E_1\wedge E_3=\delta E_2, 
$$
Also, using the fixed convention, we get the non zero components of the Riemann curvature tensor, as follows:
\begin{align}\label{Rthtt}
 \Rt(E_1,E_2)E_1=-3\tq E_2, \quad \Rt(E_1,E_3)E_1=\tq E_3,\quad \Rt(E_2,E_3)E_2=-\delta\tq E_3.
\end{align}
Moreover, using $\kappa=0$ in Proposition $2.6$ of \cite{CP1} we find that the tensor $\Rt$ can be described as in the following result.

\begin{prop}\label{prop1htt}
The Riemann curvature tensor $\Rt$ of $\htt$ is determined by
\begin{equation}\label{tensRhtt}
\begin{split}
        \Rt(X,Y)Z=&3\tq[\gt(Y,Z)X-\gt(X,Z)Y]\\
        &-\delta 4\tq[\gt(Y,E_3)\gt(Z,E_3)X- \gt(X,E_3)\gt(Z,E_3)Y\\
        &+\gt(X,E_3)\gt(Y,Z)E_3-\gt(Y,E_3)\gt(X,Z)E_3],
        \end{split}
\end{equation}
for all vector fields $X,Y,Z$ on $\htt$.
\end{prop}

\subsection{Surfaces in Three-Dimensional Lorentzian Manifolds}\label{geosurfaces}

Consider an isometric immersion $F: M \to \tilde M$, where $(M, g)$ is a pseudo-Riemannian surface inside a three-dimensional pseudo-Riemannian manifold $(\tilde M, \tilde g)$. The metric $g$ on $M$ is induced by $\tilde g$ through $F$.

Let $N$ be a unit normal vector field along $\Sigma$, satisfying $\tilde g(N, N) = \varepsilon$, where $\varepsilon \in {-1, 1}$. If $\varepsilon = 1$, $N$ is spacelike and $M$ is timelike; if $\varepsilon = -1$, $N$ is timelike and $M$ is spacelike.

Denote by $\nabla$ and $\tilde \nabla$ the Levi-Civita connections of $M$ and $\tilde M$. For any tangent vector fields $X, Y$ on $M$, the \textit{Gauss formula} yields
\begin{align} \label{fG}
\tilde \nabla_X Y = \nabla_X Y + h(X,Y)N,
\end{align}
where $h$ is the \textit{second fundamental form}, a symmetric bilinear form on $TM$.

Moreover, the {\em Weingarten formula} reads:
\begin{align}\label{W}
SX=-\tilde \nabla_X N,
\end{align}
where $S$ denotes the shape operator of $M$, which is related to $h$ in the following way:
$$
\tilde g(SX, Y)=\e h(X,Y)
$$
for all $X$, $Y$ tangent to $M$.

Special properties of the second fundamental form (or, equivalently of the shape operator) characterize special classes of surfaces. If $h=0$, (equivalently, if $S$ vanishes) $M$ is called \textit{totally geodesic}, meaning that its geodesics remain geodesics in $\tilde M$. In addition, a surface is \textit{totally umbilical}, that is,
\begin{align} \label{TU}
h(X,Y) = \lambda g(X,Y)
\end{align}
for some smooth function $\lambda$ locally defined on $M$.
In terms of the the shape operator $S$, \eqref{TU} is equivalent to  the requirement
$$
S = \lambda \, \text{Id}.
$$
Another significant family of surface is that of \textit{parallel surfaces}, defined by the condition
\begin{align*}
\nabla h = 0,
\end{align*}
which implies that the extrinsic invariants of the surface are covariantly constant. Clearly, parallel surfaces are totally geodesic.

In \cite{CV1} the authors characterized parallel surfaces in three-dimensional Lorentzian manifold by proving the following.

\begin{lemma}[\cite{CV1}]\label{lem}
Let $(M, g)$ be a surface in a three-dimensional Lorentzian manifold,
$N$ an $\e$-unit normal to $M$ and $\{E_1, E_2\}$ a pseudo-orthonormal frame field with $g(E_1, E_1) = 1$, $g(E_1, E_2) = 0$ and $g(E_2, E_2)=-\e$, such that the shape operator
associated to $N$ takes the form
$$
S =\begin{pmatrix}
S_{11} & S_{12}\\
-\e S_{12} & S_{22}
\end{pmatrix},
$$
with respect to $\{E_1, E_2\}$. Then $M$ is parallel if and only if

\begin{equation}\label{parallel}
\begin{cases}
X(S_{11}) = -2\e S_{12} g(\nabla_X E_1, E_2),\\
X(S_{12}) = (S_{22}- S_{11}) g(\nabla_X E_1, E_2),\\
X(S_{22}) = 2\e S_{12} g(\nabla_X E_1, E_2),
\end{cases}
\end{equation}
for every tangent vector $X$.
\end{lemma}

Next, using the second fundamental form we can express the {\em mean curvaure}, that is

$$
H = \frac{1}{2} \mathrm{tr}_{g} h = \frac{1}{2} \sum g^{ij}h_{ij}.
$$

If we take into account the shape operator we obtain the following equivalent definition of the mean curvature:

\begin{equation}\label{H}
H = \frac{1}{2} \mathrm{tr} S.
\end{equation}

If $H = 0$, the surface $M$ is called {\em  minimal}. (If $M$ is timelike, it is sometimes referred to as maximal instead.) If $H$ is constant, the surface is said to have constant mean curvature ({\em CMC}, for short). Every totally geodesic surface is minimal, and by definition, every minimal surface is a CMC surface.

A simple calculation using Lemma \ref{lem} and the expression of the mean curvature given in \eqref{H} proves the following.

\begin{prop}\label{pcmc}
Any parallel surface in a three-dimensional Lorentzian manifold has constant mean curvature.
\end{prop}
\begin{proof}
For a parallel surfaces equations \eqref{parallel} hold; in particular
$$
\begin{cases}
X(S_{11}) = -2\e S_{12} g(\nabla_X E_1, E_2),\\
X(S_{22}) = 2\e S_{12} g(\nabla_X E_1, E_2).
\end{cases}
$$
Next, using \eqref{H} we shall calculate
$$
X(H)=\frac{1}{2} X(\mathrm{tr} S)=\frac{1}{2} \left( X(S_{11})+X(S_{22})\right)=0
$$
for every tangent vector $X$, so  that $H$ is constant on $M$.
\end{proof}

Observe that, if $M$ is a totally umbilical surface, then $S_{11}=S_{22}=\lambda$ and $S_{12}=0$. Requiring that $M$ is CMC is equivalent to restrict ourselves to a constant $\lambda$.
Thus, a simple substitution in \eqref{parallel} proves the following.

\begin{cor}\label{corTU}
Let $M$ be a totally umbilical surface in a three-dimensional Lorentzian manifold; then $M$ has constant mean curvature if and only if it is parallel.
\end{cor}

Not all classes of surfaces are defined starting from either $h$ or $S$. In particular, it is sufficient to establish some function on an oriented surfaces to obtain some interesting geometric properties. For example, a well known class of surfaces is defined by the constancy of the angle the normal forms with a given direction of the ambient space. To be more precise, if we now consider an oriented pseudo-Riemannian surface $M$ immersed into $\tilde{M}$ and a Killing vector field $\tilde{V}$, we can define the \textit{angle function}
\begin{align*}
\nu:=\tilde{g}(N, \tilde{V})\tilde{g}(N,N)=\e\tilde{g}(N, \tilde{V}),
\end{align*}
where $N$ is the $\e$-unit normal to $M$. A surface is said to be a {\em constant angle surface} or a {\em helix surface} if $\nu$ is a real constant.

\section{The structure equations for surfaces in $\htt$}
In this section, let us consider a pseudo-Riemannian oriented surface $M$ immersed into $\htt$ and give the following:
\begin{defn}
The {\em angle function}, or simply {\em angle}, is the function defined by:
\begin{align*}
\nu:=\gt(N,E_3)\e,
\end{align*}
where $N$ is the $\e$-unit normal to $M$, that is, $\gt(N,N)=\e=\pm 1$.
\end{defn}    
Next, we consider the decomposition $E_3=T+\nu N$ to obtain
    \begin{align} \label{ottohtt}
        \gt(T,T)=\delta-\nu^2 \e.
    \end{align}
This leads to:
    \begin{align*}
        \nt _XE_3&=\nt _X T+ X(\nu)N+\nu \nt _X N=\\
        &=\nabla _X T+ \e \gt(SX,T)N+X(\nu)N-\nu SX.
    \end{align*}
Moreover, by \eqref{nxe3htt} we get:
\begin{align*}
 \nt_X E_3=\delta \tau X\wedge E_3=\e \delta\tau  \gt(JX,T)N-\delta \tau \nu JX,
\end{align*}
where we called $JX:=N\wedge X$, the rotation of angle $\pi/2$ in $TM$, which satisfies the relations $\gt(JX,JX)=-\e\gt(X,X)$ and $J^2X=\e X$.

Comparing the two expressions we obtain:
    \begin{align}\label{36htt}
    \begin{cases}
        \nabla_X T=\nu (SX-\delta\tau JX) ,\\[3pt] X(\nu)=-\e \gt(SX-\delta \tau JX, T).
        \end{cases}
    \end{align}
    
Next, we will give the expressions of the Gauss and Codazzi equations for a pseudo-Riemannian surface $M$ immersed into $\htt$:

\begin{prop}\label{gchtt}
Let $X,Y$ denote vector fields tangent to $M$, $K$ the Gaussian curvature of  $M$ and $\bar{K}$ the sectional  curvature  in $\htt$ of the plane tangent to $M$. Then,
\begin{equation}\label{kkhtt}
      K=\bar{K}+\lm \det S=-\tq +\lm\left[\det S+4\delta \nu^2\tq\right]
\end{equation}
and
\begin{equation}\label{Codhtt}
\nabla_X SY-\nabla_YSX-S[X,Y]=-4\delta \e \nu \tq \left[\gt(X,T)Y-\gt(Y,T)X\right].
\end{equation}
\end{prop}
\begin{proof}
We start recalling the the Gauss equation for pseudo-Riemannian surfaces immersed into $\htt$:
\begin{align}\label{geqprs}
    K=\bar{K}+\lm \frac{\gt(SX,X)\gt(SY,Y)-\gt(SX,Y)^2}{\gt(X,X)\gt(Y,Y)-\gt(X,Y)^2}
\end{align}
where, if $X$ and $Y$ are orthonormal, we have that:
\begin{align*}
\frac{\gt(SX,X)\gt(SY,Y)-\gt(SX,Y)^2}{\gt(X,X)\gt(Y,Y)-\gt(X,Y)^2}=\det S.
\end{align*}
Now, considering a local (pseudo-)orthonormal basis $\{ X,Y\}$ on $M$, i.e., $\gt(X,X)=1$, $\gt(X,Y)=0$ and $\gt(Y,Y)=-\lm$, we have
\begin{align*}
\bar{K}(X,Y)=-\e \gt(\Rt(X,Y)Y,X)=\tq(-1+4\e\delta \nu^2),
\end{align*}
whence equation \eqref{kkhtt} follows.

We now consider the Codazzi equation for hypersurfaces, that is
\begin{align}\label{ceqhs}
     \gt(\Rt(X,Y)Z,N)=\gt(\nabla_XA(Y)-\nabla_YA(X)-A[X,Y],Z).
\end{align}

Therefore, since
\begin{align*}
\Rt(X,Y)N=4\delta\e\tq\nu[\gt(X,T)Y-\gt(Y,T)X],
\end{align*}
we get equation \eqref{Codhtt} by the arbitrarity of $Z$.
\end{proof}

\section{Minimal helix surfaces in $\htt$}
Let us assume that the surface $M$ forms a constant angle with the vector field $E_3$, that is, $\nu$ is a real constant and so $M$ is a helix surface. We prove the following.

\begin{prop}\label{prop4.1}
Let $M$ denote a helix surface in $\htt$ and $N$ the $\e$-unit vector field normal to $M$ such that $\gt(N,N)=\e$. Then, if we consider the tangent basis  $\{T, JT\}$ where $T$ is the tangent projection of $E_3$ on $M$ and $JT=N\wedge T$, the following hold:
    \begin{itemize}
        \item[(i)] the shape operator of $M$ is given by
        \begin{align*}
           S= \begin{pmatrix}
                    0 &\delta\e \tau \\
                    -\delta\tau &\mu
            \end{pmatrix}
        \end{align*}
for some smooth function $\mu$ on $M$;
    \item[(ii)] the Levi-Civita connection $\nabla$ of $M$ is described by
    \begin{align*}
            \nabla_T T= -2\delta \tau \nu\, JT, \qquad \nabla_{JT}T=\mu  \nu\,  JT, \\
            \nabla_T JT= -2\e\delta \tau \nu\,  T, \qquad \nabla_{JT}JT=\e\mu  \nu\,  T;
     \end{align*}
     \item[(iii)] the Gaussian curvature of $M$ is constant and is given by
	   \begin{equation}
        K=4\delta \e \tq \nu^2 ;
    \end{equation}
      \item[(iv)]
      the function $\mu$ satisfies the equation
    \begin{equation}\label{star}
        T(\mu)+\mu^2 \nu -4\delta \tq \nu^3 =0.
    \end{equation}
    \end{itemize}

\end{prop}
\begin{proof}
Considering the tangent basis $\{T, JT \}$ and using opportunily \eqref{36htt}, we get \textit{(i)} and \textit{(ii)}.

We now proceed calculating the  Gaussian curvature. By \eqref{kkhtt} we find
\begin{align*}
    K=-\tq +\e\left[\det S+4\delta \nu^2\tq\right]=4\delta \e \tq \nu^2
\end{align*}

Finally, we calculate
\begin{align}\label{A}
\begin{array}{ll}
 &\nabla_TS(JT)-\nabla_{JT}S(T)-S[T,JT]=\\
         &=\nabla_T(\delta \e\tau T+\mu JT)-\nabla_{JT}(-\delta\tau JT)- S\left(-2\e\delta\nu\tau T-\mu\nu JT\right)=\\
         &= \left( -4\e \tq \nu + T(\mu) + \mu^2 \nu \right)JT.
\end{array}
\end{align}
By Proposition \ref{gchtt}, we then get
\begin{align}\label{B}
\begin{array}{ll}
    \nabla_TS(JT)-\nabla_{JT}S(T)-S[T,JT]&=-4\delta \e \nu \tq [\gt(T,T)JT-\gt(JT,T)T]=\\
    &=\left(-4\e \nu \tq +4 \delta\tq \nu^3 \right) JT
 \end{array}
\end{align}
and so, by comparing \eqref{A} and \eqref{B}, we get \eqref{star}.
\end{proof}

\begin{rem}
We already know from \cite{CCLP} that the Lorentzian Heisenberg group does not admit any totally umbilical (in particular, totally geodesic) surface. For the case of helix surfaces, this is clear observing that $\tau\neq 0$ in the matricial form of the shape operator given in Proposition {\em \ref{prop4.1}}.
\end{rem}

Since we are interested in describing some minimal and CMC surfaces in $\htt$, we shall prove the following.

\begin{cor}\label{cor}
Let $M$ be a helix surfaces in $\htt$. Following the notation introduced in Proposition {\em\ref{prop4.1}},  $M$ has constant mean curvature $H$ if and only if $\mu=2H$.

In particular, the helix surface $M$ is minimal if and only if $\mu=0$.
\end{cor}
\begin{proof}
The statement follows from the calculation of the mean curvature of $M$ using \eqref{H} and the expression of $S$ given in Proposition \ref{prop4.1}.
\end{proof}

\begin{rem}\label{nu0}
Observe that using $\mu=0$ in \eqref{star} we deduce $\nu=0$; however, in general the converse is not true. In fact, there exist some surfaces with constant angle $\nu=0$ that are not minimal.

Moreover, if we require that $\mu\neq 0$ is constant (i.e. $M$ is proper CMC) and we put $\mu=2H$ in \eqref{star}, we deduce that
\begin{itemize}
\item either $\nu=0$, or
\item  $\mu^2=4\delta\tq \nu^2$.
\end{itemize}
In particular, the second case occurs only if $\delta=1$, and in this case $\mu=2\tau \nu$, that is, $M$ is a CMC, helix surface with $H=\tau \nu$ prescribed by the geometry of the space.
\end{rem}

We point out that for a general constant angle surface $M$ with $\nu=0$, the vector field $E_3$ is fully tangent to $M$, which means that the $\e$-unit normal vector field can be given by $N= aE_1+bE_2$ for some smooth function $a$ and $b$ defined on  some open neighbourhood $U$ of $M$ such that $a^2-\delta b^2=\e$. In particular, without any other assumption, it is not immediate to obtain a significant and explicit parametrization for such a surface. In any case, using $\nu=0$ in (ii) of Proposition \ref{prop4.1} we deduce that $M$ is flat; moreover, more geometrically we observe that $M$ is some tubular surface around the $E_3$-direction. 

\medskip
 We shall now  describe helix surfaces with $\nu=0$ obtaining explicit parametrization for the minimal and CMC cases. As $\e$ and $\delta$ are not fixed, in order to simplify the calculation, we shall treate separately two distinct cases, depending on the sign of $\delta$.

\medskip
\textbf{Case (I): $\boldsymbol{\delta =-1}$.} In this case, as $\e =\gt(N,N)=a^2+ b^2=1$ the surface $M$ is timelike and there exists a smooth function $\varphi: U \rightarrow \R$ such that $a=\cos \varphi$ and $b=\sin \varphi$.

Next, the vector fields
\begin{align}
\label{frameI}
T=E_3, \qquad JT=N\wedge T= \sin \varphi E_1-\cos \varphi E_2, \qquad 
\end{align}
span the tangent space to $M$ at every point. Moreover, using $\nu=0$ in Proposition \ref{prop4.1} we obtain that
$$
\nabla_T T=\nabla_{JT} T=\nabla_T JT=\nabla_{JT} JT=0
$$
thus $M$ is flat and vector fields
\begin{equation}\label{cooI}
T  = \pa {u}, \quad JT  = \pa {v}
\end{equation}
may be taken as coordinate vector fields on $M$, as they commute.

Moreover, using \eqref{LCivitahtt} and \eqref{frameI} it follows that
$$
\begin{array}{l}
\tau =\dfrac{\gt(AT, JT)}{\gt(JT,JT)}=-\gt\left(\nt_{T} N, JT\right)=\varphi (u,v)_u+\tau,
\\[7pt]
\mu=\dfrac{\gt(AJT, JT)}{\gt(JT,JT)}=-\gt\left(\nt_{JT} N, JT\right)=\varphi(u,v)_v ,
\end{array}
$$
that leads to $\mu=\mu(v)$, as $\varphi(u,v)=\varphi(v)$ such that $\varphi'(v)=\mu(v)$.

Next, from Corollary \ref{cor} we deduce that $M$ has constant mean curvature $H$ if and only if 
$$
\varphi(v)=2H v+\varphi_0
$$
for some real constant $\varphi_0$. In particular, $M$ is minimal if and only if $H=0$, that is, $\varphi=\varphi_0$.

\medskip
When $\varphi=\varphi_0$, denote by $F: M \rightarrow \htt, \; (u,v)\mapsto(x(u,v), y(u,v), z(u,v))$ the immersion of the minimal surface $M$ in the local coordinates introduced in \eqref{cooI}. By \eqref{frameI} and \eqref{E123} we obtain
\begin{equation} \label{eqpaI}
\begin{array}{ll}
(\pa {u} x, \pa {u} y, \pa {u} z)=& (0,0,1),\\[6pt]
(\pa {v} x, \pa {v} y, \pa {v} z)=& \left(\sin \varphi_0
,-\cos \varphi_0
,-\tau \left(\cos \varphi_0
 \, x(u,v)+\sin \varphi_0
\, y(u,v)\right)\right).
\end{array}
\end{equation}
Then, by integration of \eqref{eqpaI}, we find
$$
F(u,v)=\left(\sin \varphi_0 v+c_1
,-\cos \varphi_0v +c_2 
,u +c_3\right)
$$
for some real constants $c_1$, $c_2$ and $c_3$. Thus, since translations are isometries of the ambient space we obtain the immersion
$$
F(u,v)=\left(\sin \varphi_0\, v,-\cos \varphi_0\,v, u\right).
$$

Observe that in the global coordinates $(x,y,z)$ of $H_3$, the  flat, minimal helix surface $M$, as described above, is characterized by the Cartesian equation 
$$
\sin \varphi_0\, y + \cos \varphi_0\, x=0;
$$
so, $M$ may be interpreted as (a part of) a \lq\lq plane\rq\rq \ in $H_3$, that is clearly a special case of a tube.

\medskip
When $H\neq0$, denote again by $F: M \rightarrow \htt, \; (u,v)\mapsto(x(u,v), y(u,v), z(u,v))$ the immersion of the proper CMC surface $M$ in the local coordinates introduced above. Using  \eqref{frameI},  and \eqref{cooI}, we obtain
$$
\begin{array}{ll}
(\pa {u} x, \pa {u} y, \pa {u} z)=& (0,0,1),\\[6pt]
(\pa {v} x, \pa {v} y, \pa {v} z)=& \left(\sin \varphi(v)
,-\cos \varphi(v)
,-\tau \left(\cos \varphi(v)\,
 x(u,v)+\sin \varphi(v)\,
y(u,v)\right)\right).
\end{array}
$$
After the reparametrization using the map $\varphi(v)=2H v+\varphi_0\mapsto v$, by integration, we get
$$
F(u,v)=\left(-\cos v, -\sin v,u -\tau v +c_3\right)
$$
for some real constant $c_3$. Therefore, up to isometries of the ambient space, the immersion is given by
$$
F(u,v)=\left(-\cos v, -\sin v,u -\tau v\right).
$$

Let us observe that in the global coordinates $(x,y,z)$ of $H_3$, the  flat, CMC, helix surface $M$ given above may be interpreted as (a part of) a \lq\lq cylinder\rq\rq \ as it is characterized by the Cartesian equation 
$$
x^2+y^2=1.
$$

\medskip
\textbf{Case (II): $\boldsymbol{\delta=1}$.}  As $N=aE_1+bE_2$, so $a^2-b^2=\e=\pm1$, we shall consider separately two cases, depending on the sign of $\e$. In fact, if $\varepsilon=1$ (respectively, $\varepsilon=-1$), then $M$ is timelike (resp. spacelike) and there exists a smooth function $\varphi:U \rightarrow \R$ such that $a=\cosh \varphi$ (resp. $\sinh \varphi$) and $b=\sinh \varphi$ (resp. $\cosh \varphi$).

In both cases, the vector fields
\begin{align}
\label{frameII}
T=E_3, \qquad JT=N\wedge T= b(\varphi) E_1+a(\varphi) E_2
\end{align}
span the tangent space to $M$ at every point and again, using $\nu=0$ in Proposition \ref{prop4.1} we deduce
$$
\nabla_T T=\nabla_{JT} T=\nabla_T JT=\nabla_{JT} JT=0.
$$
Hence, $M$ is flat and vector fields
\begin{align}
\label{cooII}
\pa u=T, \qquad \pa v=JT
\end{align}
can be taken as coordinate vector fields on $M$, so that the functions $\mu=\mu(u,v)$ and $\varphi=\varphi(u,v)$. Next, in analogy with the previous case, after computing the shape operator of the surface $M$ using the definition \ref{W} and comparing with the one given in Proposition \ref{prop4.1}, we deduce that $\mu(v)=-\varphi'(v)$. Moreover, from Corollary \ref{cor}, we get that $M$ has constant mean curvature $H$ if and only if 
$$
\varphi(v)=-2H v+\varphi_0
$$
for some real constant $\varphi_0$. In particular, $M$ is minimal if and only if $H=0$, that is $\varphi=\varphi_0$.

Next, denote by $F: M \rightarrow \htt, \; (u,v)\mapsto(x(u,v), y(u,v), z(u,v))$ the immersion of the constant angle surface $M$ in the local coordinates introduced in \eqref{cooII}. By \eqref{frameII} and \eqref{E123} we obtain
\begin{equation} \label{eqpaII}
\begin{array}{ll}
(\pa {u} x, \pa {u} y, \pa {u} z)=& (0,0,1),\\[6pt]
(\pa {v} x, \pa {v} y, \pa {v} z)=& \left(b(\varphi(v))
,a (\varphi(v))
,\tau \left(-a(\varphi(v))
 x(u,v)+b(\varphi(v))
y(u,v)\right)\right).
\end{array}
\end{equation}

\medskip 
Then, when $M$ is minimal (i.e. $H=0$) we obtain, by integration of \eqref{eqpaII}, the following parametrizations
\begin{equation}\label{min2}
F(u,v)=\begin{cases}
\left(\sinh \varphi_0 \, v,\cosh \varphi_0 \, v,u \right) \quad \mathrm{if}\,\, \e=1,\\[6pt]
\left(\cosh \varphi_0 \, v,\sinh \varphi_0 \, v,u \right) \quad \mathrm{if}\,\, \e=-1
\end{cases}
\end{equation}
where we used isometries of the ambient space.

\medskip
When $H\neq0$, we can use the reparametrization $\varphi(v)=-2H v+\varphi_0\mapsto v$; moreover, by integration of \eqref{eqpaII} and applying isometries of the ambient space, we get
\begin{equation}\label{cmc2}
F(u,v)=\begin{cases}
\left(\cosh v,\sinh v,u -\tau v\right) \quad \mathrm{if}\,\, \e=1,\\[6pt]
\left(\sinh  v,\cosh v, u+\tau v \right) \quad \mathrm{if}\,\, \e=-1.
\end{cases}
\end{equation}

The above calculations are summarized in the following main result of this section.

Again, as in Case (I), any of the minimal (respectively CMC) surfaces parametrized by \eqref{min2} (respectively, \eqref{cmc2}) may be interpreted as (a part of) a either Lorentzian  or Riemannian \lq\lq plane \rq\rq (respectively, \lq\lq hyperboloid\rq\rq ).

\begin{theorem}\label{main}
Let $M$ be a constant angle surface immersed into the Lorentzian Heisenberg group $\htt$. If $M$ is minimal, then it is some either Riemannian or Lorentzian flat \lq\lq plane\rq\rq  \ corresponding to one of the following cases:
\begin{itemize}
\item[(I)] if $\delta=-1$, $M$ is timelike and it is parametrized by
$$
F(u,v)=\left(\sin \varphi_0	\, v,-\cos \varphi_0\, v, u\right);
$$
\item[(II)] if $\delta=1$, $M$ is either timelike or spacelike and it is parametrized by
$$
F(u,v)=\begin{cases}
\left(\sinh \varphi_0	\, v,\cosh \varphi_0\, v,u \right) \quad \mathrm{(timelike)}, \\[6pt]
\left(\cosh \varphi_0\,  v,\sinh \varphi_0\, v,u \right) \quad \mathrm{(spacelike)};
\end{cases}
$$
\end{itemize}
where $\varphi_0$ is a real constant.

Moreover, the Lorentzian Heisenberg group $\htt$ admits as proper {\em CMC} surfaces any open subset of some (either Riemannian or Lorentzian) flat \rq\rq cylinder\lq\lq  \ $M$ corresponding to one of the following cases:
\begin{itemize}
\item[(I)] if $\delta=-1$, $M$ is timelike and it is parametrized by
$$
F(u,v)=\left(-\cos v, -\sin v,u -\tau v\right);
$$
\item[(II)] if $\delta=1$, $M$ is either timelike or spacelike and it is parametrized by
$$
F(u,v)=\begin{cases}
\left(\cosh v,\sinh v,u -\tau v\right) \quad \mathrm{(timelike)},\\[6pt]
\left(\sinh  v,\cosh v, u+\tau v \right) \quad \mathrm{(spacelike)}.
\end{cases}
$$
\end{itemize}
\end{theorem}

At the end of this section, let us observe that the equivalence given in Corollary \ref{corTU} is not significant in $\htt$ as it is known (see \cite{CCLP}) that there are no totally umbilical surfaces in $\htt$. We shall then prove the following.

\begin{cor}\label{cor0}
Let $M$ be a constant angle surface with $\nu=0$ in $\htt$; then $M$ has constant mean curvature if and only if it is parallel.
\end{cor}
\begin{proof}
From Corollary \ref{cor} we deduce that parallel $0$-angle surfaces are CMC, that is $\mu$ is constant. Conversely, if $\mu$ is constant, condition \eqref{parallel} hold, as $\nabla$ vanishes on a $0$-angle surface $M$ in both cases for $\delta=\pm 1$, therefore $M$ is paralellel.  
\end{proof}

\section{General helix surfaces in $\htt$}

In the rest of the paper we shall study helix surfaces in $\htt$ for some constant angle $\nu\neq 0$.

As we pointed out in the Introduction, constant angle surfaces in $\htt$ have been already studied by Onnis and Piu in \cite{OP}; however, their study did not cover all the possible cases for Lorentzian metrics on $H_3$. In particular, if we consider the metric given in \eqref{gonH}, they studied the case  corresponding to $\delta=-1$ (with a parameter $\tau$ that has an opposite sign).

For this reason, from now on, we shall consider always $\delta=1$ and we refer the reader to \cite{OP} for  the corresponding result for the \lq\lq more classical\footnote{The case for $\delta=-1$ is more classical in the sense that it is the restriction for $\kappa=0$ of the LBCV spaces studied firstly by Yildirim \cite{Y}, that is, the ones obtained by a Lorentzian fibration over a Riemannian surface.}\rq\rq \ case.

\subsection{Spacelike helix surfaces in $\htt$}
In this section we begin the study of general (i.e. not minimal) helix surfaces in $\htt$. Firstly let us consider the spacelike case, where $\e=-1$, therefore, from the equation \eqref{ottohtt} it follows that (up to the orientation of $N$) we can write $\nu= \sinh \vartheta\neq 0$, where $\vartheta > 0$ is called the \textit{hyperbolic angle function} between $N$ and $E_3$. As a consequence we obtain $\gt(T,T)=\cosh^2\vartheta\neq 0$.

Now, using $\delta=1$, $\e=-1$ and $\nu= \sinh \vartheta$ in Proposition \ref{prop4.1} we obtain that with respect to the tangent basis  $\{T, JT\}$, the matrix describing the shape operator is given by
        \begin{align}\label{S_space}
           S= \begin{pmatrix}
                    0 &-\tau\\
                    -\tau &\mu
            \end{pmatrix},
        \end{align}
for some smooth function $\mu$ on $M$ satisfying the equation
    \begin{equation}\label{starhtt}
        T(\mu)+\mu^2 \sinh \vartheta -4\tq \sinh^3 \vartheta =0.
    \end{equation}
    
Moreover, the Levi-Civita connection $\nabla$ of $M$ is described by
    \begin{align*}
            \nabla_T T= -2\tau \sinh \vartheta\, JT, \qquad \nabla_{JT}T=\mu  \sinh \vartheta\,  JT, \\
            \nabla_T JT= 2\tau \sinh \vartheta\,  T, \qquad \nabla_{JT}JT=- \mu  \sinh \vartheta\,  T;
     \end{align*}
and the Gaussian curvature of $M$ is constant and is given by
$$
K=-4\tq  \sinh^2 \vartheta.
$$

As we know that $\gt(E_3, N)=- \sh $ and that $\delta=1$, there exists a smooth function $\varphi$ on $M$ such that
 \begin{equation}\label{NN}
  N=\ch \left(\sph E_1+ \cph E_2\right)-\sh E_3,
\end{equation}
then:
    \begin{align}\label{TS}
        &T=E_3-\sh N=-\sh\ch \left(\sph E_1+ \cph E_2 \right)+\cosh^2\vartheta E_3,\\
        &JT=N\wedge T=N\wedge E_3=\ch (\cph E_1 + \sph E_2).\label{JTS}
    \end{align}
Moreover, using \eqref{LCivitahtt} wit the expressions given in  \eqref{NN}, \eqref{TS} and \eqref{JTS} we can prove the following
    \begin{align*}
        &ST=-\nt_T N=-\left(T(\varphi)+\tau \cosh^2 \vartheta + \tau \sinh^2 \vartheta\right) \,JT,\\
        &SJT=-\nt_{JT}N=-JT(\varphi)\,JT-\tau\, T
    \end{align*}
and, comparing with \eqref{S_space}, we get:
    \begin{align}\label{TJT}
    \begin{cases}
    T(\varphi)=-2\tau\sinh^2\vartheta,\\
    JT(\varphi)=-\mu.
    \end{cases}
    \end{align}
Next, observe that the compatibility equation, namely
$$
[T,JT](\varphi)=\left(\nabla_T JT-\nabla_{JT}T\right)(\varphi)
$$
is equivalent to \eqref{starhtt}.

As we are interested in describing explicitly the surface, we now choose local coordinates $(u,v)$ on $M$ such that
    \begin{align}\label{coo1htt}
        \partial_u=T, \quad \partial_v=aT+bJT
    \end{align}
where $a, b$ are smooth functions on $M$. The condition $0=[\partial_u,\partial_v]$ leads to:
    \begin{align}\label{abeqs}
        \begin{cases}
            a(u,v)_u=-2\tau\,\sh\, b(u,v)\\
            b(u,v)_u=\mu\, \sh\, b(u,v) .
        \end{cases}
    \end{align}
In conclusion, integrating \eqref{starhtt} we get
$$
        \mu(u,v)=2\tau \, \sh \, \tanh\left(2\tau \sinh^2 \vartheta \, u+\eta(v)\right).
$$

\begin{rem}\label{impo}
Observe that, if we require $\nu\neq 0$, from above we deduce that $\mu$ is never constant, i.e. there are no examples of spacelike CMC helix surfaces with $\nu\neq 0$.
\end{rem}

Now, since we only need one solution for \eqref{abeqs}, we can take for example
    \begin{align}\label{coo2htt}
        \begin{cases}
            a(u,v)=-\frac{\sinh \left(2\tau \sinh^2 \vartheta \, u+\eta(v)\right)}{\sh}\\
            b(u,v)=\cosh\left(2\tau \sinh^2 \vartheta \, u+\eta(v)\right).
        \end{cases}
    \end{align}
Therefore, from \eqref{TJT} we obtain
$$
        \begin{cases}
            \varphi_u(u,v)=-2\tau \sinh^2 \vartheta , \\
            \varphi_v(u,v)=0,
        \end{cases}
$$
that lead to $\varphi(u,v)=\varphi(u)=-2\tau \sinh^2 \vartheta \, u+c$, where $c$ is a real constant. We now prove the following.

\begin{theorem}

\label{ppvhtt}
Let $M$ be a helix spacelike surface in $\htt$
with constant hyperbolic angle $\vartheta$. Then, with respect to the local coordinates $(u,v)$ on $M$ defined in \eqref{coo1htt} and \eqref{coo2htt}, the position vector $F$ of $M$ in $\R^3$ is given by
\begin{align}\label{pvhtt}
\begin{split}
  F(u,v)=\Bigl (& \frac{\coth \vartheta}{2\tau} \cosh (u)+f_1(v),\frac{\coth \vartheta}{2\tau} \sinh (u)+f_2(v),\\ &\,\,\frac{\cosh^2\vartheta}{2} u-\frac{\coth \vartheta}{2\tau}\left(f_2(v)\cosh (u)-f_1(v)\sinh (u)\right) +f_3(v) \Bigr),
\end{split}
\end{align}
where $f_1, f_2, f_3$ satisfy:
\begin{align*}
f'_1(v)^2-f'_2(v)^2=\cosh^2 \vartheta, \qquad f'_3(v)=\tau \left(f_1(v)f'_2(v)-f_2(v)f'_1(v)\right).
\end{align*}
\end{theorem}
\begin{proof}
Let $F: \Omega \rightarrow \htt$, $(u,v)\mapsto F(u,v)=(x(u,v), y(u,v), z(u,v))$ denote the (local) immersion of the helix spacelike surface $M$ with local coordinates $(u,v)$ introduced in \eqref{coo1htt} and defined on an open set $\Omega$ of $\R^2$.

By \eqref{TS}, \eqref{JTS} and \eqref{coo2htt} we get
$$
\begin{array}{rl}
\partial_u F=&T=-\sh\ch \left(\sph (u)\, E_{1|F}+ \cph (u)\, E_{2|F} \right)+\cosh^2\vartheta \,E_{3|F},\\[6pt]
\partial_v F=& a(u,v)\, T+ b(u,v)\, JT=\\
	=&-a(u,v)\, \sh\ch \left(\sph (u)\, E_{1|F}+ \cph(u)\, E_{2|F} \right)+\cosh^2\vartheta\, E_{3|F}\\
        &+ b(u,v)\, \ch \left(\cph (u)\,E_{1|F} + \sph (u)\,E_{2|F}\right).
\end{array}
$$

Moreover, using the expressions for $E_1$, $E_2$ and $E_3$ in terms of the coordinate basis $\{\pa x, \pa y, \pa z\}$ of $\R^3$ as specified in \eqref{E123}, we can calculate
    \begin{align}\label{deuFhtt}
        \begin{array}{l}
        x(u,v)_u=- \sh \ch \sph(u)\\
        y(u,v)_u=- \sh \ch \cph(u)\\
        z(u,v)_u=\ch \left(\tau \sh \sph (u) y(u,v) - \tau \sh \cph(u) x(u,v)+\ch \right)
        \end{array}
    \end{align}
and
    \begin{align}\label{devFhtt}
        \begin{array}{rl}
        x(u,v)_v=\ch &\left(-a(u,v)\sh \sph (u) + b(u,v)\cph (u)\right)\\
        y(u,v)_v=\ch &\left(-a(u,v)\sh \cph (u) + b(u,v)\sph (u)\right)\\
        z(u,v)_v=\ch  &\Big[-\tau y(u,v)\left(-a(u,v)\sh \sph (u) +b(u,v) \cph (u)\right)\\
        &+ \tau x(u,v)\left( -a(u,v)\sh \cph (u) +b(u,v) \sph (u) \right)\\
        &+a(u,v) \ch \Big]
        \end{array}
    \end{align}
Then, by integration of \eqref{deuFhtt} we obtain 
\begin{align*}
\begin{array}{l}
x(u,v)=\dfrac{\coth \vartheta}{2\tau} \cph (u)+f_1(v),\\[6pt]
y(u,v)=\dfrac{\coth \vartheta}{2\tau} \sph (u)+f_2(v),\\[6pt]
z(u,v)=\dfrac{\cosh^2\vartheta}{2} u-\dfrac{\coth \vartheta}{2\tau}\left(f_2(v)\cph (u)-f_1(v)\sph (u)\right) +f_3(v),
\end{array}
\end{align*}
where $f_i(v)$, $i=1, 2, 3$ are some smooth functions defined on $M$ that are determined by the relations found in \eqref{devFhtt}, that is,
\begin{align}\label{32htt}
\begin{array}{rl}
f'_1(v)&=\ch \left(\sinh \left(2\tau \sinh^2 \vartheta \, u+\eta(v)\right) \sph (u) + \cosh\left(2\tau \sinh^2 \vartheta \, u+\eta(v)\right)\cph (u)\right)\\
&=\ch\, \cosh\left(\eta(v)+c \right),\\[6pt]
f'_2(v)&=\ch \left(\sinh \left(2\tau \sinh^2 \vartheta \, u+\eta(v)\right)\cph (u) + \cosh\left(2\tau \sinh^2 \vartheta \, u+\eta(v)\right)\sph (u)\right)\\
&=\ch\, \sinh\left(\eta(v)+c \right),\\[6pt]
f'_3(v)&=\coth\vartheta \left(f'_2(v)\cph(u)-f'_1(v)\sph(u)\right)+\tau \left(f_1(v)f'_2(v)-f_2(v)f'_1(v)\right)+a(u,v)\cosh^2\vartheta\\
&=\tau \left(f_1(v)f'_2(v)-f_2(v)f'_1(v)\right).
\end{array}
\end{align}

Therefore, one obtains equation \eqref{pvhtt} after the reparametrization $\varphi(u)\mapsto u$.
\end{proof}


\subsection{Timelike helix surfaces in $\htt$}
We now study timelike helix surfaces following the same approach used in the spacelike case. In this case, where $\e=1$, from the equation \eqref{ottohtt} it follows that (up to the orientation of $N$) we can write $\nu= \sin \vartheta\neq 0$, where $\vartheta$ is called, analogously, the \textit{hyperbolic angle function} between $N$ and $E_3$. Now, let us assume again that $\vartheta$ is constant.

Again, by mean of opportune substitution of parameters in Proposition \ref{prop4.1} we get that, with respect to the tangent basis  $\{T, JT\}$, the matrix describing the shape operator of the timelike helix surface $M$ is given by
        \begin{align}\label{S_time}
           S= \begin{pmatrix}
                    0 & \tau\\
                    -\tau &\mu
            \end{pmatrix},
        \end{align}
for some smooth function $\mu$ on $M$ that satisfies equation
    \begin{equation}\label{starbhtt}
        T(\mu)+\mu^2 \sin\vartheta-4\tq \sin^3 \vartheta =0.
    \end{equation}
The Levi-Civita connection $\nabla$ of $M$ is described by
    \begin{align*}
            \nabla_T T= -2\tau \sin\vartheta\, JT, \qquad \nabla_{JT}T=\mu  \sin\vartheta\,  JT, \\
            \nabla_T JT= - 2\tau \sin\vartheta\,  T, \qquad \nabla_{JT}JT= \mu  \sin\vartheta\,   T;
     \end{align*}
and the Gaussian curvature of $M$ is constant and is given by
$
K=4\tq  \sin^2 \vartheta;
$

Again, as we know that $\gt(E_3, N)= \sin\vartheta$; considering the causal character of $E_1$, $E_2$ and $E_3$ we obtain that there exists a smooth function $\varphi$ on $M$ such that:
    \begin{align}
        N=\cos\vartheta \left(\cph\, E_1+ \sph\, E_2\right)+\sin\vartheta\, E_3,
    \end{align}
then:
    \begin{align}\label{TT}
        &T=E_3-\sin\vartheta N=-\sin\vartheta \cos\vartheta \left(\cph\, E_1+\sph\, E_2\right)+ \cos^2\vartheta\, E_3,\\\label{JTT}
        &JT=N\wedge T=N \wedge E_3=\cos\vartheta \left(\sph\, E_1+ \cph\, E_2 . \right.
    \end{align}
Moreover, we can prove the following
    \begin{align*}
        &ST=-\nt_T N=-\left(T(\varphi)+\tau \cos^2 \vartheta - \tau \sin^2 \vartheta\right)\,JT,\\
        &SJT=-\nt_{JT}N=-JT(\varphi)\,JT+\tau\, T,
    \end{align*}
and, comparing with \eqref{S_time}, we find:
    \begin{align}\label{TJTt}
    \begin{cases}
        JT(\varphi)=-\mu \\ T(\varphi)=2\tau \sin^2 \vartheta,
    \end{cases}
    \end{align}
whose compatibility is equivalent to \eqref{starbhtt}.
We now choose local coordinates $(u,v)$ on $M$ as in \ref{coo1htt}, then the condition $0=[\partial_u,\partial_v]$ now leads to:
    \begin{align}
        \begin{cases}
            a(u,v)_u=2\tau \sin \vartheta\, b(u,v), \\
            b(u,v)_v=\mu(u,v)\, \sin \vartheta	\, b(u,v).
        \end{cases}
    \end{align}
In conclusion, integrating \eqref{starbhtt} we get
$$
        \mu(u,v)=2\tau \, \sin \vartheta\, \tanh\left(2\tau \sin^2 \vartheta \, u+\eta(v)\right).
$$

\begin{rem}\label{remCMC}
Observe again that, if we require $\nu\neq 0$, we obtain that $\mu$ is never constant, that is, there are no examples of timelike helix surfaces with $\nu\neq 0$ are never CMC. Therefore, from Remarks {\em \ref{nu0}} and {\em \ref{impo}}, we deduce that proper CMC constant angle surfaces described in Theorem {\em  \ref{main}} are the only CMC surfaces in the class of constant angle surfaces, i.e. the condition $\nu=0$ is necessary for the constancy of the mean curvature in this class of submanifolds.
\end{rem}

Now, since we are searching again just for one solution, let us take for example
    \begin{align}\label{coo2bhtt}
        \begin{cases}
            a(u,v)=-\dfrac{\sinh\left(2\tau \sin^2 \vartheta \, u+\eta(v)\right)}{\sin\vartheta}\\[6pt]
            b(u,v)=\cosh\left(2\tau \sin^2 \vartheta \, u+\eta(v)\right).
        \end{cases}
    \end{align}
Therefore we get
$$
        \begin{cases}
            \varphi(u,v)_u=2\tau \sin^2 \vartheta\\
            \varphi(u,v)_v=0
        \end{cases}
$$
that leads to $\varphi(u,v)=\varphi(u)2\tau \sin^2 \vartheta\, u+c$, where $c\in \mathbb{R}$ is a constant.

\begin{theorem}\label{ppvbhtt}
Let $M$ be a helix timelike surface in $\htt$
with constant hyperbolic angle $\vartheta$. Then, with respect to the local coordinates $(u,v)$ on $M$ defined in \eqref{coo1htt} and \eqref{coo2bhtt}, the position vector $F$ of $M$ in $\R^3$ is given by
\begin{align}\label{pvbhtt}
\begin{split}
  F(u,v)=\Bigl (&-\dfrac{\cot \vartheta}{2\tau} \sinh(u)+f_1(v),-\dfrac{\cot \vartheta}{2\tau} \cosh (u)+f_2(v),\\ & \dfrac{\cos^2\vartheta}{2} u-\dfrac{\cot \vartheta}{2\tau}\left(f_1(v)\cosh (u)-f_2(v)\sinh (u)\right) +f_3(v) \Bigr),
\end{split}
\end{align}
where $f_1, f_2, f_3$ satisfy:
\begin{align*}
f'_1(v)^2-f'_2(v)^2=-\cos^2 \vartheta, \qquad f'_3(v)=\tau \left(f_2(v)f'_1(v)-f_1(v)f'_2(v)\right).
\end{align*}
\end{theorem}
\begin{proof}
Let now  $F: \Omega \rightarrow \htt$, $(u,v)\mapsto F(u,v)=(x(u,v), y(u,v), z(u,v))$ denote the (local) immersion of the helix timelike surface $M$ with local coordinates $(u,v)$ introduced above. Next, using again \eqref{E123} and \eqref{TT}, \eqref{JTT} and \eqref{coo2bhtt} we obtain
\begin{align}\label{deuFbhtt}
        \begin{array}{l}
        x(u,v)_u=-  \sin \vartheta \cos \vartheta \cph(u)\\
        y(u,v)_u=- \sin \vartheta \cos \vartheta \sph(u)\\
        z(u,v)_u=\cos \vartheta \left(\tau \sin\vartheta \cph (u) y(u,v) - \tau \sin\vartheta \sph(u) x(u,v)+ \cos \vartheta \right)
        \end{array}
    \end{align}
and
    \begin{align}\label{devFbhtt}
        \begin{array}{rl}
        x(u,v)_v=\cos \vartheta &\left(-a(u,v)\sin\vartheta \cph (u) + b(u,v)\sph (u)\right)\\
        y(u,v)_v=\cos \vartheta &\left(-a(u,v)\sin\vartheta \sph (u) + b(u,v)\cph (u)\right)\\
        z(u,v)_v=\cos \vartheta  &\Big[-\tau y(u,v)\left(-a(u,v)\sin\vartheta \cph (u) + b(u,v)\sph (u)\right)\\
        &+ \tau x(u,v)\left(-a(u,v)\sin\vartheta \sph (u) + b(u,v)\cph (u)\right)\\
        &+a(u,v) \cos \vartheta \Big]
        \end{array}
    \end{align}
Therefore, integrating \eqref{deuFbhtt}, we get
\begin{align*}
\begin{array}{l}
x(u,v)=-\dfrac{\cot \vartheta}{2\tau} \sph (u)+f_1(v),\\[6pt]
y(u,v)=-\dfrac{\cot \vartheta}{2\tau} \cph (u)+f_2(v),\\[6pt]
z(u,v)=\dfrac{\cos^2\vartheta}{2} u-\dfrac{\cot \vartheta}{2\tau}\left(f_1(v)\cph (u)-f_2(v)\sph (u)\right) +f_3(v),
\end{array}
\end{align*}
where $f_i(v)$, $i=1, 2, 3$ are some smooth functions defined on $M$ that are determined by the relations found in \eqref{devFbhtt}, namely
\begin{align}\label{32bhtt}
\begin{array}{rl}
f'_1(v)&=-\cos \vartheta\, \sinh\left(\eta(v)-c \right),\\[6pt]
f'_2(v)&=\cos \vartheta\, \cosh\left(\eta(v)-c \right),\\[6pt]
f'_3(v)&=\tau \left(f_1(v)f'_2(v)-f_2(v)f'_1(v)\right).
\end{array}
\end{align}

Again, one obtains equation \eqref{pvbhtt} after the reparametrization $\varphi(u)\mapsto u$.
\end{proof}


In conclusion, let us recall Remark \ref{remCMC}. Observe that, also for analogous surfaces described in \cite{OP}, since there $\delta=-1$ any constant angle surface that has a constant mean curvature shows, if exists, an angle $\nu=0$ (see Remark \ref{nu0}); thus we give the following general result.

\begin{theorem}
Let $M$ be a constant angle surface in $\htt$; then $M$ has constant mean curvature if and only if it is parallel.
\end{theorem}
\begin{proof}
From Remark \ref{remCMC} and applying the above arguments to surfaces in \cite{OP} we deduce that for CMC constant angle surface the angle $\nu$ vanishes. To conclude, as parallel surfaces are CMC and Corollary \ref{cor0} holds for surfaces with constant angle $\nu=0$, we get the thesis.
\end{proof}

\subsection*{Acknowledgment}
The author would like to thank Prof. Giovanni Calvaruso for the valuable suggestions in the preparation of the manuscript.


\begin{thebibliography}{9999}

\bibitem{CCLP}
G. Calvaruso, M. Castrill\'on-Lopez and L. Pellegrino, {\em On totally umbilical and minimal surfaces of the Lorentzian Heisenberg groups}, Math. Nachr., {\bf 298} (2025), 1922--1942.

\bibitem{COPU}
G. Calvaruso, I.I. Onnis, L. Pellegrino and D. Uccheddu, {\em Helix surfaces for Berger-like metrics on the anti-de Sitter space}, Rev. Real Acad. Cienc. Exactas Fis. Nat. Ser. A-Mat., {\bf 118} (2024), 54, 29 pp.

\bibitem{CP0}
G. Calvaruso and L. Pellegrino, {\em Totally umbilical surfaces of Lorentzian reducible spaces}, J. Geom. Anal., {\bf 35} (2025) , 131, 22 pp.

\bibitem{CP1}
G. Calvaruso and L. Pellegrino, {\em Lorentzian BCV spaces: properties and geometry of their surfaces}, {J. Austral. Math. Soc.}, (2025), 23pp.

\bibitem{CP2}
G. Calvaruso and L. Pellegrino, {\em Surfaces of 3D homogeneous plane waves}, J. Geom. Phys, 217 (2025), 12pp.

\bibitem{CP3}
G. Calvaruso and L. Pellegrino, {\em On surfaces of exceptional Lorentzian Lie groups with a four-dimensional isometry group},  Mathematics, 13 (2025), 14pp.


\bibitem{CV1}
G. Calvaruso and J. Van der Veken, \emph{Parallel surfaces in three-dimensional Lorentzian Lie groups}, Taiwanese J. Math., \textbf{14} (2010), 223--250. 

\bibitem{CV4}
G. Calvaruso and J. Van der Veken, \emph{Parallel surfaces in three-dimensional reducible spaces}, Proc. Roy. Soc. Edinburgh Sect. A, \textbf{143} (2013), 483--491. 

\bibitem{CD}
P. Cermelli and A.J. Di Scala, \emph{Constant-angle surfaces in liquid crystals}, Phil. Mag., \textbf{87} (2007), 1871--1888.

\bibitem{DFVV} 
F. Dillen, J. Fastenakels, J. Van der Veken, and L. Vrancken, \emph{Constant angle surfaces in $\mathbb S^2 \times \mathbb R$}, Monatsh. Math.,  \textbf{152} (2007), 89--96.

\bibitem{DM}
F. Dillen and M.I. Munteanu, \emph{Constant angle surfaces in $\mathbb H^2 \times \mathbb R$}, Bull. Braz. Math. Soc. \textbf{40} (2009), 85--97.

\bibitem{DMVV}
F. Dillen, M.I. Munteanu, J. Van der Veken, and L. Vrancken, \emph{Classification of constant angle surfaces in a warped product}, 
Balkan J. Geom. Appl., \textbf{16} (2011), 35--47.

\bibitem{DR}
A. Di Scala and G. Ruiz-Hernandez, \emph{Helix submanifolds of Euclidean spaces}, Monatsh. Math., \textbf{157} (2009), 205--215.


\bibitem{DR2}
A. Di Scala and G. Ruiz-Hernandez, \emph{Higher codimensional Euclidean helix submanifolds}, Kodai Math. J., \textbf{33} (2010), 
192--210.


\bibitem{FMV}
J. Fastenakels, M.I. Munteanu, and J. Van Der Veken, \emph{Constant angle surfaces in the Heisenberg group}, Acta Math. sinica, \textbf{27} (2011), 747--756.

\bibitem{FN}
Y. Fu and A.I. Nistor, {\em Constant Angle Property and Canonical Principal Directions for Surfaces in $\mathbb{M}^2(c)\times \R_1$}, Mediterr. J. Math., \textbf{10} (2013), 1035--1049.

\bibitem{IV}
J. Inoguchi, J. Van der Veken, {\em A complete classification of parallel surfaces in three-dimensional homogeneous spaces},Geom. Dedicata, \textbf{131} (2008), 159--172. 


\bibitem{LM}
R. Lopez and M.I. Munteanu, \emph{On the geometry of constant angle surfaces in $Sol_3$}, Kyushu J. Math., \textbf{65} (2011), 
237--249.

\bibitem{LM2}
R. Lopez and M.I. Munteanu, \emph{Constant angle surfaces in Minkowski space}, Bull. Belg. Math. Soc. Simon Stevin, \textbf{18} 
(2011), 271--286.

\bibitem{LO}
P. Lucas and J.A. Ortega-Yag\"ues, \emph{Helix surfaces and slant helices in the three-dimensional anti-De Sitter space}, RACSAM, \textbf{111} (2017), 1201--1222.

\bibitem{MO}
S. Montaldo and I.I. Onnis, \emph{Helix surfaces in the Berger sphere}, Israel J. Math., \textbf{201} (2014), 949--966.

\bibitem{MOP} S. Montaldo, I.I. Onnis and A.P. Passamani, \emph{Helix surfaces in the special linear group}, Ann. Mat. Pura Appl.  \textbf{195} (2016),  59--77.


\bibitem{MN}
M.I. Munteanu and A.I. Nistor, \emph{A new approach on constant angle surfaces in $\mathbb{E}^3$}, Turkish J. Math., \textbf{33}, (2009), 2, 1--10.
 
\bibitem{NiA}
A.I. Nistor, \emph{Certain constant angle surfaces constructed on curves}, Int. Electron. J. Geom. \textbf{2} (2011), 79--87.


\bibitem{Ni}
A.I. Nistor, \emph{Constant angle surfaces in solvable Lie groups}, Kyushu J. Math., \textbf{68} (2014), 315--332.


\bibitem{OP}
I.I. Onnis and P. Piu, \emph{Constant angle surfaces in the Lorentzian Heisenberg group}, Arch. Math., \textbf{109} (2017), 575--589.

\bibitem{OPP}
I.I. Onnis, A.P. Passamani and P. Piu, \emph{Constant Angle Surfaces in Lorentzian Berger Spheres}, J. Geom. Anal., \textbf{29} (2019), 1456--1478.

\bibitem{P}
L. Pellegrino, {\em Constant Angle Surfaces in $\mathbb{M}^2_1(\kappa)\times \R$}, Mediterr. J. Math.,  {\bf 101} (2025), 22 pp.


\bibitem{RR1}
N. Rahmani and S. Rahmani, {\em Structures homogenes lorentziennes sur le groupe de Heisenberg. I}, J. Geom. Phys.,  
{\bf 13} (1994), 254--258.

\bibitem{RR2}
N. Rahmani and S. Rahmani, {\em Lorentzian Geometry of the Heisenberg Group}, Geom. Dedicata,  {\bf 118} (2006), 133--140.


\bibitem{Ru}
G. Ruiz-Hernandez, \emph{Minimal helix surfaces in $N^n \times \mathbb R$}, Abh. Math. Semin. Univ. Hamburgh, \textbf{81} (2011), 
55--67.


\bibitem{TA}
E. Turhan and G. Altay, {\em Minimal surfaces in three dimensional Lorentzian Heisenberg group}, Beitr\"age Alg. Geom., 
\textbf{55}  (2014), 1--23.

\bibitem{Y}
A. Yildirim, {\em On Lorentzian BCV spaces}, Int. J. Math. Archive, \textbf{3} (2012), 1365--1371.

\end{thebibliography}
\end{document}